\documentclass[11pt]{amsart}
\usepackage{tkz-euclide}
\usepackage{latexsym}
\usepackage{amssymb,amsmath,mathrsfs}
\usepackage{cases}
\usepackage{amsthm}
\usepackage{hyperref}
\usepackage{enumitem}
\usepackage[capitalise]{cleveref}
\usepackage{mathtools}
\makeatletter
\newenvironment{mcases}[1][l]
 {\let\@ifnextchar\new@ifnextchar
  \left\lbrace
  \array{@{}l@{\quad}#1@{}}}
 {\endarray\right.}
\makeatother

\allowdisplaybreaks %allow aligned equations to split up over pages%
\newtheorem{theorem}{Theorem}[section]
\newtheorem{definition}{Definition}[section]

\newtheorem{prop}[theorem]{Proposition}
\theoremstyle{definition}

\makeatletter
\def\th@remark{%
  \thm@headfont{\bfseries}%
  \normalfont % body font
  \thm@preskip\topsep \divide\thm@preskip\tw@
  \thm@postskip\thm@preskip
}
\makeatother

\theoremstyle{remark}

\newtheorem{lemma}[theorem]{Lemma}\newtheorem{innercustomlemma}{Lemma}
\newenvironment{customlemma}[1]
  {\renewcommand\theinnercustomlemma{#1}\innercustomlemma}
  {\endinnercustomlemma}

\newtheorem{innercustomthm}{Theorem}
\newenvironment{customtheorem}[1]
  {\renewcommand\theinnercustomthm{#1}\innercustomthm}
  {\endinnercustomthm}

\newcommand{\Hilbert}{\mathbb{H}}

\newcommand{\half}{(-\overline{\Delta})^\frac{1}{2}}

\newcommand{\grad}{\overline{\nabla}}
\newcommand{\lap}{\overline{\Delta}}

\newcommand{\dnu}{\partial_\nu \Psi}

\newcommand{\lambdan}{\sqrt{\lambda_n}}

\def\Xint#1{\mathchoice
{\XXint\displaystyle\textstyle{#1}}%
{\XXint\textstyle\scriptstyle{#1}}%
{\XXint\scriptstyle\scriptscriptstyle{#1}}%
{\XXint\scriptscriptstyle\scriptscriptstyle{#1}}%
\!\int}
\def\XXint#1#2#3{{\setbox0=\hbox{$#1{#2#3}{\int}$ }
\vcenter{\hbox{$#2#3$ }}\kern-.6\wd0}}

\def\dashint{\Xint-}

\newcommand{\supp}{\operatorname{supp}}
\newcommand{\riesz}{\mathcal{R}}

\title{Classical Solutions for the 3D Quasi-Geostrophic System on a Bounded Domain}
\author[Novack]{Matthew D. Novack}
\author[Vasseur]{Alexis F. Vasseur}
\address[Matthew D. Novack]{\newline Department of Mathematics, \newline The University of Texas at Austin, Austin, TX 78712, USA}
\email{mnovack@math.utexas.edu}
\address[Alexis F. Vasseur]{\newline Department of Mathematics, \newline The University of Texas at Austin, Austin, TX 78712, USA}
\email{vasseur@math.utexas.edu}

\date{\today}

\thanks{\textbf{Acknowledgment}: The second author was partially funded by the NSF during this work}

\begin{document}
\begin{abstract}
We revisit a model for three-dimensional, inviscid quasi-geostrophic flow on bounded, cylindrical domains introduced by the authors in \cite{nv18}.  We prove the local-in-time existence of classical solutions.
\end{abstract}
\maketitle \centerline{\date}
\section{Introduction}
The quasi-geostrophic system is a set of equations used to describe oceanic and atmospheric motion over large time-scales.  Much of the existing literature treats the case of a physical boundary at the (top and) bottom of the domain while specifying that the horizontal variables $(x,y)$ belong to either $\Omega = \mathbb{R}^2, \mathbb{T}^2$. In these cases, the equations take the form
\[\begin{dcases}
        \left( \partial_t + \grad^\perp \Psi\cdot \grad \right) \left( \mathcal{L}(\Psi) + \beta_0 y \right)= 0 & \Omega\times[0,h]\times[0,T] \\
        \left( \partial_t + \grad^\perp \Psi \cdot \grad \right) (\partial_{\nu}\Psi) = 0 & \Omega\times\{0,h\}\times[0,T].\\
       \end{dcases}   \qquad (QG)
\]
The normal derivative of $\Psi$ on $\Omega\times\{0,h\}$ is denoted by $\dnu$, while $\grad = (\partial_x, \partial_y, 0)$, and $\grad^\perp = (-\partial_y, \partial_x, 0)$. The operator $\mathcal{L}$ is defined by
$$ \mathcal{L}:=\partial_{xx}+\partial_{yy}+ \partial_z \left( \lambda \partial_z \right) $$
where $\lambda>0$ is a smooth function depending only on $z$ and is related to the density of the fluid. To ensure ellipticity of $\mathcal{L}$ one requires that $$\frac{1}{\Lambda} \leq \lambda(z) \leq \Lambda$$ for some $\Lambda \in (0,\infty)$.

In a recent work \cite{nv18}, the fully three-dimensional system was considered on a domain with non-trivial lateral boundary conditions for the first time.  Using only the assumption that the fluid velocity $\grad^\perp\Psi$ does not penetrate the boundary $\partial\Omega\times[0,h]$ (i.e. $\grad^\perp\Psi\cdot \nu_s = 0$), the following model was derived:
\[ 
\begin{dcases}
       \left( \partial_t + \grad^\perp\Psi\cdot\grad \right) \left( \mathcal{L}(\Psi) + \beta_0 y \right) =0 &  \Omega\times(0,h)\times[0,T] \\
       \left( \partial_t + \grad^\perp\Psi\cdot\grad \right) \partial_\nu \Psi =0 &  \Omega\times\{0,h\}\times[0,T]\\
       \Psi(x,y,z,t) = \Psi(z,t) & \partial\Omega\times[0,h]\times[0,T] \\
       \frac{\partial}{\partial t}\dashint_{\partial\Omega\times\{z\}} \grad \Psi(z,t) \cdot \nu_s = 0 & [0,h]\times[0,T].
       \end{dcases}
\]
The quantity
\begin{equation}\label{jofz}
\dashint_{\partial\Omega\times\{z\}} \grad \Psi(z,t) \cdot \nu_s = \dashint_{\partial\Omega\times\{z\}} \grad \Psi(z,0) \cdot \nu_s =: j(z)
\end{equation}
is therefore a datum that must be prescribed along with an initial vorticity $f=\mathcal{L}(\Psi_0)$ and Neumann condition $g=\partial_\nu \Psi_0$. We emphasize however that we \emph{do not} prescribe the lateral boundary values $\Psi(z,t)$.  They instead arise as the boundary conditions naturally dual to \eqref{jofz} when solving an elliptic problem in an appropriate Hilbert space (see \cref{Hilbertspace}). In \cite{nv18}, we proved the existence of a global weak solution to this system.  

In this paper we prove the existence of a classical solution to this system for smooth enough initial data on a short time interval.

\begin{theorem}\label{shorttimeexistence}
Given $\lambda$, $f$, $g$, $j$ such that
\begin{enumerate}
    \item  $\lambda$ is $C^5$, there exists $\Lambda$ such that $\frac{1}{\Lambda} \leq \lambda \leq \Lambda$, and $\frac{\partial^k}{\partial z^k} \lambda|_{z=0,h} = 0$ for $k=1,2,3$
    \item $f\in H^4(\Omega\times(0,h))$ and $f$ vanishes in a neighborhood of $\partial\Omega\times\{0,h\}$
    \item $g\in \bar{H}^4(\Omega\times\{0,h\})$ and $g$ is compactly supported in $\Omega\times\{0,h\}$
    \item $j\in H^4(0,h)$ and $\frac{\partial^k}{\partial z^k} j(z) |_{z=0,h}= 0$ for $k=1,3$
\end{enumerate}
there exists a time $T_0$ which for large data satisfies
$$T_0 \gtrsim \left( \|f\|_{H^4} + \|g\|_{\bar{H}^4} + \|j\|_{H^4} \right)^{-1}$$
such that QG posed on the cylindrical domain $\Omega\times[0,h]$ has a classical solution on the time interval $[0,T_0]$.
\end{theorem}

The compatibility conditions appear necessary for the construction of smooth solutions.  Essentially, we can only treat data $f$, $g$ which are zero in a neighborhood of the corners $\partial\Omega\times\{0,h\}$ and datum $j$ which retain smoothness after even reflection over the boundaries at $z=0,h$. While we can prove higher order elliptic regularity under weaker conditions on the data $f$, $g$, $j$, and $\lambda$ (see \cref{scompatibility}), it is not clear that these conditions are preserved by the evolution of the system (see \eqref{compatibilityinaction} on page 11).  However, an important consideration is that our conditions are still sufficiently general to treat a broad class of initial data for the special case of 2D SQG (see the next subsection for a discussion of the relation of our model to proposed models for 2D SQG on bounded domains).

\subsection{Previous Results}
Quasi-geostrophic flow is an asymptotic limit of 3D Navier-Stokes or Euler equations as the Rossby number $\epsilon \rightarrow 0$. Study of the inviscid three dimensional QG system was initiated in the absence of boundary conditions by Bourgeois and Beale \cite{bb}.  With the aid of viscosity at the boundary, Desjardins and Grenier built global-in-time weak solutions \cite{dg}.  Both of the afore-mentioned papers also include proofs that on the interval of time for which a smooth solution to the limiting system persists, the solutions to Navier-Stokes/Euler converge to the solution to QG.  In \cite{pv}, Puel and the second author introduced a reformulation of the inviscid 3D QG system in terms of $\nabla\Psi$ which allowed for the construction of weak solutions via compactness.  Using again the reformulation, the first author then extended this existence result to a wider class of initial data and determined the conditions under which the energy $\| \nabla \Psi(t) \|_{L^2}$ is conserved \cite{novackweak}.  In a recent work, the first author also addressed the case of energy-dissipative weak solutions via a convex integration argument \cite{novack19}.  Global regularity to the 3D model with dissipation was shown by the authors in \cite{novackvasseur}.

A special case of the three-dimensional model called the surface quasi-geostrophic equation arises by specifying that $\beta_0 =0$, $\lambda =1$, and $\Delta\Psi|_{t=0}=0$.  Then the stream function $\Psi$ remains harmonic for all times $t>0$, in which case the dynamics can be described completely by the evolution of $\partial_\nu \Psi$.  Since $\Psi$ is harmonic, one has that $\partial_\nu\Psi|_{z=0} = \half \Psi|_{z=0} =: \theta$ and $\grad^\perp\Psi|_{z=0} = \riesz^\perp \theta =: u$ where $\riesz$ is the vector of two dimensional Riesz transforms, and thus the equation can be written as 
$$  \partial_t \theta + u \cdot \grad \theta = 0.  $$
Study of 2D SQG began with the work of Constantin, Majda, and Tabak \cite{cmt}.  Weak solutions were constructed by Resnick \cite{Resnick}, with an extension of the theory by Marchand \cite{Marchand}.  The equivalences of the various notions of weak solutions to 2D SQG and 3D QG were shown by the first author in \cite{novackweak}.  In the presence of a viscous term $\half \theta$, global regularity of 2D SQG has been shown by Kiselev, Nazarov, and Volberg \cite{knv}, Caffarelli and the second author \cite{cv}, Constantin and Vicol \cite{cvicol}, and Kiselev and Nazarov \cite{Kiselev2010}.  

One way of approaching two-dimensional quasi-geostrophic dynamics on a smooth, bounded set $\Omega\subset \mathbb{R}^2$ is to specify a notion of Riesz transform in order to define the velocity $u=\riesz^\perp\theta$.  A natural choice is to define the half-Laplacian spectrally using the Dirichlet eigenfunctions, an approach initiated by Constantin and Ignatova in \cite{ci} and \cite{ci2}.  Further work by Constantin and Nguyen \cite{cn2}, \cite{cn}, Nguyen \cite{nguyen18}, and Constantin, Nguyen, and Ignatova \cite{cin} has explored further questions concerning this model.  However, it is not hard to see that solutions to inviscid SQG constructed using the spectral Riesz transform and extended harmonically to $z>0$ cannot coincide with the solutions to 3D QG we produce in this paper.  The difference lies in the lateral boundary conditions.  The use of the Dirichlet Laplacian $\lap_\Omega$ imposes that the extended stream function 
$$\Psi|_{\partial_\Omega\times[0,h]}=0.$$ 
However, the lateral boundary values of our stream function are not uniformly zero.  Furthermore, in the introduction of \cite{nv18}, we show that solutions constructed via the spectral Riesz transform do not satisfy \eqref{jofz} either.  Therefore, one of the main motivations of this work was to validate the physical relevance of the three-dimensional model and associated lateral boundary conditions derived in \cite{nv18}.

The outline for this paper is as follows.  In the next section, we first provide an intuition for the elliptic problem in the simple case that $\Omega$ is a ball. We then recall previous results and prove the higher regularity estimates needed to construct classical solutions.  In the third section, we construct classical solutions using a fixed-point argument.  The appendix contains a short justification for the use of a commutator estimate in our setting which is classical for $\mathbb{T}^n$ or $\mathbb{R}^n$.

\section{A Non-Standard Elliptic Problem}

\subsection{A Simple Case}

As described previously, building solutions to QG on a cylinder requires a choice of datum $j(z)$ which encodes the "average Neumann condition" as
\begin{equation}
\dashint_{\partial\Omega\times\{z\}} \grad \Psi(z,t) \cdot \nu_s =  j(z). \nonumber
\end{equation}
With this choice, reconstrucing $\Psi(t)$ at a given time can be done by solving an elliptic problem using $\Delta\Psi(t)$, $\partial_\nu\Psi(t)$, and $j(z)$ as data.  This elliptic problem takes the form
\[
(E) = \begin{dcases}
       \mathcal{L}(u) = f &  \Omega\times[0,h]\\
       \partial_{\nu} u = g & \Omega \times \{0,h\} \\
       u(x,y,z) = u(z) & \partial\Omega \times [0,h]\\
       \int_{\partial\Omega\times\{z\}}  \grad u\cdot \nu_s =j(z) & [0,h]. 
\end{dcases}
\]
As alluded to before, we cannot choose $u(z)$; rather, it arises as the condition naturally dual to the average Neumann datum $j(z)$. Let us suppose now that $\Omega$ is the unit ball so that we have access to rotational symmetries. Define $\tilde{u}$ to be the rotational average of $u$
$$ \tilde{u} = \int_{0}^{2\pi} u(r,\theta,z) \,d\theta,  $$
and set $u' = u - \tilde{u}$.  Assuming for the time being that we can solve the elliptic problem $(E)$, what is the equation satisfied by $\tilde{u}$?  Due to the fact that $\mathcal{L}$ commutes with rotations in $\theta$, we see that
$$  \mathcal{L}(\tilde{u}) = \int_0^{2\pi} f(r,\theta,z)\,d\theta  =: \tilde{f}.  $$
In addition, we will also have that $\partial_\nu \tilde{u} = \tilde{g}$, where $\tilde{g}$ is the rotational average of $g$.  Finally, as $j$ is invariant under rotations and $\partial_{\nu_s} \tilde{u}$ depends only on $z$, we find that $\tilde{u}$ solves the Neumann problem
\[
\begin{dcases}
       \mathcal{L}(\tilde{u}) = \tilde{f} &  \Omega\times[0,h]\\
       \partial_{\nu} \tilde{u} = \tilde{g} & \Omega \times \{0,h\} \\
       \tilde{u}(x,y,z) = u(z) & \partial\Omega \times [0,h]\\
       \grad \tilde{u}\cdot \nu_s = j(z) & \partial\Omega\times [0,h]. 
\end{dcases}
\]
By linearity, $u' = u - \tilde{u}$ then solves
\[
\begin{dcases}
       \mathcal{L}(u') = f-\tilde{f} &  \Omega\times[0,h]\\
       \partial_{\nu} u' = g-\tilde{g} & \Omega \times \{0,h\} \\
       u'(x,y,z) = 0 & \partial\Omega \times [0,h]. 
\end{dcases}
\]
That is, $\tilde{u}$ encodes the rotationally symmetric portion of $u$ and solves a Neumann problem, while $u'$ encodes the deviations from rotational symmetry and solves an elliptic problem with Neumann data on the top and bottom and Dirichlet data on the lateral boundaries. We remark that even in this simplified setting, it is evident that the data must satisfy some compatibility conditions at the corners in order for $u$ to be smooth.

\subsection{Previous Results and Definitions}

Following \cite{nv18}, we define the Hilbert space to which the solution $\Psi(t)$ will belong for each time $t$.

\begin{definition}\label{Hilbertspace}
Define $H$ by
$$ H := \left\{ \alpha \in C^\infty\left(\bar{\Omega}\times[0,h]\right): \quad \int_{\Omega\times[0,h]}\alpha \,dx\,dy\,dz=0, \quad  \alpha|_{\partial\Omega \times [0,h]}(x,y,z)=\alpha(z) \right\}.$$
Using the notation $\tilde{\nabla}=(\partial_x,\partial_y, \lambda(z) \partial_z)$, equip $H$ with the inner product
$$ \langle \alpha, \gamma \rangle_\Hilbert := \int_{\Omega\times[0,h]} \tilde{\nabla} \alpha \cdot {\nabla}{\gamma} \,dx\,dy\,dz. $$
Define the Hilbert space $\Hilbert$ as the closure of $H$ under the norm induced by this inner product.
\end{definition}

The construction of weak solutions requires a compatibility condition on the initial data due to the fact that test functions which are equal to a nontrivial constant throughout $\Omega\times(0,h)$ do not belong to $\Hilbert$. 

\begin{definition}[\textbf{Basic Compatibility}]\label{compatibility}
Any triple $(f,g,j)$ of functions with $f(x,y,z)\in L^2(\Omega\times[0,h])$, $g(x,y,z)\in L^2(\Omega\times\{0,h\})$, $j(z)\in L^2(0,h)$ is compatible if
$$\int_{\Omega\times[0,h]}f(x,y,z) \,dx\,dy\,dz = \int_0^h j(z)\,dz + \int_{\Omega\times\{0,h\}} \lambda(z)g(x,y,z) \,dx\,dy.$$
\end{definition}

For compatible data, we proved the following existence result.  

\begin{customlemma}{3.1}[\textbf{\cite{nv18}}]\label{variationalproblem}
For compatible data $(f,g,j)$, there exists a unique solution $u\in \Hilbert$ to the elliptic problem 
\[
(E) = \begin{dcases}
       \mathcal{L}(u) = f &  \Omega\times[0,h]\\
       \partial_{\nu} u = g & \Omega \times \{0,h\} \\
       u(x,y,z) = u(z) & \partial\Omega \times [0,h]\\
       \int_{\partial\Omega\times\{z\}}  \grad u\cdot \nu_s =j(z) & [0,h]. 
\end{dcases}
\]
satisfying the bound
$$  \| u \|_{\Hilbert} \leq C(\Omega,h,\lambda) \left( \|f\|_{L^2} + \|g\|_{L^2} + \|j\|_{L^2} \right).  $$
\end{customlemma}

Let $\{e_n\}_{n=1}^\infty$ and $\{\lambda_n\}_{n=1}^\infty$ be the sequence of eigenfunctions and corresponding eigenvalues for the operator $-\lap$ on $\Omega$ with homogenous Dirichlet boundary conditions; that is, 
\[
\begin{dcases}
       -\lap e_n = \lambda_n e_n &  (x,y)\in\Omega \\
       e_n = 0 & (x,y)\in\partial\Omega. \\
       \end{dcases}
\]
For $s\geq 0$, define
$$ \bar{H}^s(\Omega) = \{ g=\sum_n g_n e_n \in L^2(\Omega): \sum_n \left(\lambdan\right)^s g_n e_n \in L^2(\Omega) \}.$$
%By duality, we have that $$\left(\bar{H}^s(\Omega)\right)^* \cong \left\{ \{g_n\}_{n=1}^\infty\subset \mathbb{R}: \sum_n \frac{1}{\left(\lambdan\right)^{2s}} g_n^2 < \infty \right\}.$$
It is well known (consult section 2 of \cite{cn2} for example) that the domain of the homogenous Dirichlet Laplacian is $H^2(\Omega)\cap H^1_0(\Omega)$, and for such functions the $H^2(\Omega)$ and $\bar{H}^2(\Omega)$ norms are equivalent.  

We now define the higher-order compatibility conditions needed to prove higher regularity estimates. Data which satisfy \cref{compatibility} and \cref{scompatibility} will be called \emph{fully compatible}.

\begin{definition}[\textbf{Fully Compatible Data}]\label{scompatibility}
A triple of functions $(f,g,h)$ is  fully compatible if it is compatible (\cref{compatibility}) and satisfies in addition that
\begin{enumerate}
    \item $f\in H^{4}\left(\Omega\times(0,h)\right)$ and $\partial_z f$ vanishes on $\partial\Omega\times\{0,h\}$
    \item $g \in \bar{H}^{4}\left(\Omega\times\{0,h\}\right)$
    \item $j \in H^{4}(0,h)$ and for $k=1,3$ and $h$ the solution to
       \[
       \begin{dcases}
       \mathcal{L}(h) = f &  \Omega\times(0,h) \\
       \partial_\nu h = g & \Omega\times\{0,h\} \\
       h = 0 & \partial\Omega\times[0,h],
       \end{dcases}
       \]
    the equality
    $$ \frac{\partial^k}{\partial z^k} j \big{|}_{z=0,h} = \frac{\partial^k}{\partial z^k} \left( \int_{\Omega\times\{z\}} \grad h \cdot \nu_s \right)\bigg{|}_{z=0,h} $$
    holds.
\end{enumerate}
\end{definition}

We now recall Lemma 3.4 from \cite{nv18}.

\begin{customlemma}{3.4}{(\cite{nv18})}\label{effectofg}
Consider the equation
\[
\begin{dcases}
       \Delta u = 0 &  \Omega\times[0,h] \\
       \partial_{\nu} u = g & \Omega \times \{0,h\}\\
       u = 0 \qquad & \partial\Omega \times [0,h].\\
       \end{dcases}
\]
for $g\in\bar{H}^s(\Omega\times\{0,h\})$, $s\geq -\frac{1}{2}$. Then there exists a solution $u$ which satisfies
$$ \| \nabla u \|_{H^{s+\frac{1}{2}}(\Omega\times[0,h])}  \leq C(\Omega,h) \|g\|_{\bar{H}^s(\Omega\times\{0,h\})}  $$
\end{customlemma}

In \cite{nv18}, we proved the following elliptic regularity theorem (refer to Lemmas 3.4, 3.5, 3.6 from \cite{nv18} for the details).

\begin{customtheorem}{3.2}{(\cite{nv18})}\label{superduperelliptic}
Let $s\in[0,\frac{1}{2}]$, and let $f\in L^2(\Omega\times(0,h))$, $g\in H^s(\Omega\times\{0,h\})$, and $j\in H^s((0,h))$.  Let $u\in \Hilbert$ be the solution to $(E)$.  Then 
$$ \|\nabla u\|_{H^{s+\frac{1}{2}}(\Omega\times[0,h])} \leq C(\Omega,h,\lambda)\left( \|f\|_{L^2(\Omega\times[0,h])} + \|g\|_{H^s(\Omega\times\{0,h\})}+\|j\|_{H^s([0,h])} \right).$$ 
\end{customtheorem}

\subsection{Higher Regularity}

In order for the lateral boundary conditions to make sense, we assume the boundary of $\Omega$ has no discrete subcomponents. Higher (than $H^{5.5}$) regularity is likely available through a more careful analysis of higher order compatibility conditions.  However, the following result is satisfactory for building smooth solutions to QG and already somewhat delicate, and so we do not pursue any higher regularity here.

\begin{theorem}[\textbf{Higher Regularity}]\label{higherregularity}
Let $\Omega$ be a bounded, open set in $\mathbb{R}^2$ with a smooth ($C^\infty$, non-self-intersecting, no discrete subcomponents) boundary $\partial\Omega$.  Consider the elliptic problem
\begin{equation*}\begin{mcases}[ll@{\ }l]
       \mathcal{L}(u) = f &  \Omega\times[0,h]\\
       \partial_{\nu} u = g & \Omega \times \{0,h\}\\
       u(x,y,z) = u(z) & \partial\Omega \times [0,h]\\
       \int_{\partial\Omega\times\{z\}}  \grad u\cdot \nu_s =j(z) & [0,h]. 
\end{mcases}\end{equation*}
for a fully compatible triple of data $(f,g,j)$.  Then $u \in H^{5.5}\left(\Omega\times(0,h)\right) \cap \Hilbert$ and satisfies the bound
$$ \| u \|_{H^{5.5}} \leq C(\Omega,h,\lambda) \left( \| f \|_{H^{4}} + \| g \|_{\bar{H}^{4}} + \| j \|_{H^{4}} \right). $$
\end{theorem}

\begin{proof}
Throughout the proof, we use the notation $C(\Omega,h,\lambda)$ to describe constants that depend only on $\Omega$, $h$, $\lambda$ and may change from line to line. The proof is broken into two steps, which proceed as follows.  In Step 1, we isolate the effect of $f$ and $g$ while imposing homogenous Dirichlet conditions on $\partial\Omega\times[0,h]$.  The regularity for Step 1 proceeds via a combination of a change of variables in $z$ and bootstrapping.  By using classical elliptic regularity for $f$ and \cref{effectofg} for $g$, we obtain $H^{5.5}$ regularity. We note that we require the compatibility condition on $f$ in this first step. Then in Step 2, we analyze the effect of $j$ by reflecting over the boundaries at $z=0,h$ and utilizing \cref{superduperelliptic} and the compatibility condition between $j$ and $f$ and $g$.

\begin{enumerate}[label=Step \arabic*:]
\item Let $u_1$ be the solution to 
\[
\begin{dcases}
  \mathcal{L}(u_1) = f & \Omega\times(0,h) \\
  \partial_\nu u_1 = g & \Omega\times\{0,h\} \\
  u_1 = 0 & \partial\Omega\times[0,h].
\end{dcases}
\]
We can construct $u_1$ variationally in the subspace of $\Hilbert$ consisting of functions which vanish on $\partial\Omega\times[0,h]$.  Then $u_1$ satisfies the bound 
$$  \| u_1 \|_{\Hilbert} \leq C(\Omega,h,\lambda) \left( \| f \|_{L^2} + \| g \|_{L^2} \right). $$
Now consider 
$$\tilde{u}_1 := u_1\left(x,y,\theta(z) \right) $$
for $\theta$ solving
\[
\begin{dcases}
  \theta'(z) = \sqrt{\lambda(\theta(z))} \\      \theta(0) = 0 .
\end{dcases}
\]
By the strict positivity and smoothness of $\lambda$, $\theta:[0,h]\rightarrow [0,\tilde{h}]$ is well-defined, smooth, and a bijection for $\tilde{h}=\theta(h)$. Then we can calculate
\begin{align*}
    \Delta\left( u_1\left(x,y,\theta(z)\right)\right) &= \lap \left( u_1\left(x,y,\theta(z)\right) \right) + (\partial_{zz}u_1)\left(x,y,\theta(z)\right)(\theta'(z))^2 + (\partial_z u_1)\left(x,y,\theta(z)\right)\theta''(z)\\
    &= \lap \left( u_1\left(x,y,\theta(z)\right) \right) + (\partial_{zz}u_1)\left(x,y,\theta(z)\right)\lambda(\theta(z)) + \lambda'(\theta(z))(\partial_z u_1)\left(x,y,\theta(z)\right)\\
    &\qquad - \lambda'(\theta(z))(\partial_z u_1)\left(x,y,\theta(z)\right) + (\partial_z u_1)\left(x,y,\theta(z)\right)\theta''(z)\\
    &= f\left(x,y,\theta(z)\right)- \lambda'(\theta(z))(\partial_z u_1)\left(x,y,\theta(z)\right) + (\partial_z u_1)\left(x,y,\theta(z)\right)\theta''(z)\\
    &:= \tilde{f}(z).
\end{align*}
Notice that although $\tilde{f}$ only belongs to $L^2$ for now, $\partial_\nu \tilde{f}$ is well-defined pointwise on $\partial\Omega\times\{0,\tilde{h}\}$ and vanishes by the assumption on $f$ and the fact that $u_1 \equiv 0$ on $\partial\Omega\times[0,\tilde{h}]$. Letting $\tilde{g} = g \sqrt{\lambda(\theta(z))}$, we have shown that $\tilde{u}_1$ solves
\[
\begin{dcases}
  \Delta \tilde{u}_1 = \tilde{f} & \Omega\times(0,\tilde{h}) \\
  \partial_\nu \tilde{u}_1 = \tilde{g} & \Omega\times\{0,\tilde{h}\} \\
  \tilde{u}_1 = 0 & \partial\Omega\times[0,\tilde{h}]
\end{dcases}
\]
for $f \in L^2(\Omega\times(0,\tilde{h}))$ and $\tilde{g}\in \bar{H}^4(\Omega\times\{0,\tilde{h}\})$. 

Let $\Omega_E$ be an open, bounded set in $\mathbb{R}^3$ with smooth boundary such that $\Omega\times(0,\tilde{h})\subset\Omega_E$, and $\partial\Omega\times[0,\tilde{h}]\subset\partial\Omega_E$.  Let $f_E$ be an $L^2$ Sobolev extension of $\tilde{f}$ to $\mathbb{R}^3$ restricted to $\Omega_E$. Then consider the elliptic problem
\begin{equation*}\begin{mcases}[ll@{\ }l]
       \Delta u_2 = f_E &  \Omega_E \\
       u_2 = 0 & \partial\Omega_E. 
\end{mcases}\end{equation*}
Classical elliptic regularity theory yields that $u_2 \in H^{2}\left( \Omega_E \right)$. Note as well that since $u_2$ vanishes on $\partial\Omega_E$, $\frac{\partial^k}{\partial z^k} u_2 \equiv 0$ for any $k$ on $\partial\Omega\times\{0,h\}$.  Therefore, we can set $u_3 = \tilde{u}_1 - u_2 $ to be the solution to
\[
\begin{dcases}
  \Delta u_3 = 0 & \Omega\times(0,\tilde{h}) \\
  \partial_\nu u_3 = \tilde{g} -\partial_\nu u_2 & \Omega\times\{0,\tilde{h}\} \\
  u_3 = 0 & \partial\Omega\times[0,\tilde{h}]
\end{dcases}
\]
where $\tilde{g}-\partial_\nu u_2 \in L^2(\Omega\times\{0,\tilde{h}\})$. Then by \cref{effectofg}, $u_3 = \tilde{u}_1 - u_2 \in {H}^{1.5}(\Omega\times(0,\tilde{h}))$. Bootstrapping this estimate, we find that $\tilde{f}\in H^\frac{1}{2}(\Omega\times(0,\tilde{h}))$, and therefore $u_2 \in H^{2.5}(\Omega\times(0,\tilde{h}))$ by the same extension and restriction argument as before. Then $\partial_\nu u_2 \in \bar{H}^1(\Omega\times\{0,\tilde{h}\})$, and so $u_3 \in H^{2.5}(\Omega\times(0,\tilde{h}))$ from \cref{effectofg}. Bootstrapping again gives that $u_2 \in H^{3.5}(\Omega\times(0,\tilde{h}))$.  By the trace, $\partial_\nu u_2 \in H^2(\Omega\times\{0,\tilde{h}\})$. But since $\partial_\nu u_2$ vanishes at $\partial\Omega\times\{0,\tilde{h}\}$, $\partial_\nu u_2 \in \bar{H}^2(\Omega\times\{0,\tilde{h}\})$, and applying \cref{effectofg} again gives that $u_3 \in H^{3.5}(\Omega\times\{0,\tilde{h}\})$. Bootstrapping yet again yields $u_2 \in H^{4.5}(\Omega\times\{0,\tilde{h}\})$, and by the trace, $\partial_\nu u_2 \in H^3(\Omega\times\{0,\tilde{h}\}) \cap \bar{H}^2(\Omega\times\{0,\tilde{h}\})$.  We now claim that $\partial_\nu u_2 \in \bar{H}^{3}(\Omega\times\{0,h\})$.  For this to be true, we must show that $\lap (\partial_\nu u_2) \in \bar{H}^{1}(\Omega\times\{0,h\})$. For this we write
\begin{align*}
    \lap(\partial_\nu u_2) &= (\Delta - \partial_{zz}) \partial_\nu u_2\\
    &= \Delta (\partial_\nu u_2)\\
    &= \partial_\nu \tilde{f} \in H_0^1(\Omega\times\{0,h\}) \\
\end{align*}
by the earlier remark that $\partial_\nu \tilde{f}$ vanishes at $\partial\Omega\times\{0,\tilde{h}\}$. Therefore, we can apply \cref{effectofg} to deduce that $u_3 \in H^{4.5}(\Omega\times\{0,h\})$.  Then bootstrapping a final time with $u_2$ gives that $u_2 \in H^{5.5}(\Omega\times(0,\tilde{h}))$.  Thus we find that $\partial_\nu u_2 \in \bar{H}^4(\Omega\times\{0,\tilde{h}\})$, and we obtain that $u_3 \in H^{5.5}(\Omega\times(0,\tilde{h}))$.  

\item Now we analyze the effect of $j$. Define
$$\tilde{j}(z) = j(z) - \int_{\partial\Omega\times\{z\}}\grad u_1(z) \cdot \nu_s $$
and set $u_4 = u - u_1$.  Then $u_4$ solves
\begin{equation*}\begin{mcases}[ll@{\ }l]
       \mathcal{L} (u_4) = 0 &  \Omega\times(0,h) \\
       \partial_\nu u_4 = 0 & \Omega\times\{0,h\}\\
       u_4 = u_4(z) & \partial\Omega \times [0,h]\\
       \dashint_{\partial\Omega\times\{z\}}{\grad u_4 \cdot \nu_s} = \tilde{j}(z) & [0,h].
\end{mcases}\end{equation*}
To show higher regularity estimates on $u_4$, we will reflect over the boundaries at $z=0,h$. Let $\eta(z):[0,h)\rightarrow [0,1]$ be a smooth, compactly supported function of $z$ such that $\eta\equiv 1$ for $z\in[0,\frac{3h}{4}]$. Let $u_{4,E}:\Omega\times(-h,h)\rightarrow \mathbb{R}$ be the reflection of $\eta u_4$ over the boundary $z=0$, and let $\tilde{j}_E(z):(-h,h)\rightarrow \mathbb{R}$ be the reflection of $\tilde{j}$ over $z=0$. By the assumptions of \cref{scompatibility}, for $k=1,3$,
\begin{align*}
    \frac{\partial^k}{\partial z^k}\tilde{j}(z)|_{z=0,h} = 0.
\end{align*}
Therefore $j_E$ retains $L^2$ integrability up to derivatives of order $4$.  Here is the only point that we require the higher-order compatibility conditions on $j$.

Let us extend the operator $\mathcal{L}$ by even reflection of $\lambda(z)$ to $\lambda(|z|)$. By the assumption that $\lambda'=\lambda'''=0$ at $z=0$, we have that $\lambda(|z|)$ has well-defined derivatives up to order 4 on $[-h,h]$. One finds immediately that $\mathcal{L}(u_{4,E})=0$ for all $(x,y,z)\in\Omega\times(-h,h)$.  We now calculate $\mathcal{L}(\partial_z u_{4,E})$ by writing
\begin{align}\label{subcriticalestimate}
    \mathcal{L} ( \partial_z ( u_{4,E} ))&(x,y,z) = \left[ \mathcal{L}, \partial_z \right]( u_{4,E})(x,y,z) \nonumber \\
    & = -\lambda'(|z|)\frac{z}{|z|}(\partial_{zz}u_4(x,y,|z|) - \frac{z}{|z|} \lambda''(|z|)(\partial_{z} u_4)(x,y,|z|) \nonumber \\
    & := f_j (x,y,z).
\end{align}
We have that $f_j$ is well-defined since $\lambda'$ vanishes at 0 and $\partial_z u_4$ vanishes at 0, and
$$ \left\| f_j \right\|_{L^2(\Omega\times(-h,h))} \leq C(\Omega,h,\eta) \| u_4 \|_{H^2(\Omega\times(0,h))}. $$ 
Note that $u_4\in H^2(\Omega\times(0,h))$ by \cref{superduperelliptic}, and so this estimate makes sense. Continuing the analysis, $\partial_z (\eta u_{4,E})$ then satisfies the equation
\[ 
\begin{dcases}
       \mathcal{L} (\partial_z(\eta u_{4,E})) = f_j &  \Omega\times[-h,h] \\
       \partial_{z}\left(\partial_z(\eta u_{4,E})\right) =0 & \Omega\times\{-h,h\}\\
       \partial_z (\eta u_{4,E}) = \partial_z (\eta u_{4,E})(z) &  \partial\Omega\times[-h,h]\\
       \dashint_{\partial\Omega\times\{z\}} \grad (\partial_z (\eta u_{4,E})) \cdot \nu_s = \partial_z(\eta \tilde{j}_E) & [-h,h].
       \end{dcases}
\]
Applying \cref{superduperelliptic}, we obtain that $\nabla(\partial_{z}(\eta u_{4,E})) \in H^1(\Omega\times(-h,h))$ and satisfies the bound
$$  \left\| \nabla (\partial_z ( \eta u_{4,E} ))\right\|_{H^1(\Omega\times(-h,h))} \leq C(\Omega,h,\eta) \left( \| u_4 \|_{H^2} + \| j \|_{H^{1.5}} \right) . $$
Repeating the argument, but this time with a reflection of $(1-\eta)u_4$ over $z=h$, shows that
$$  \| \nabla ( \partial_z u_4 ) \|_{H^1(\Omega\times(0,h))} \leq C(\Omega,h,\eta) \left( \| u_4 \|_{H^2} + \| j \|_{H^{1.5}} \right).  $$

We must show that $\grad^2 u_4 \in H^1(\Omega\times(0,h))$ as well. Letting $\tau$ denote the tangent vector to $\partial\Omega$, we have that $\partial_{\tau\tau}u_4|_{\partial\Omega\times[0,h]}=0$, and therefore we can differentiate in the $\tau$ direction near $\partial\Omega$.  Then we have that
$$  \partial_{\tau^\perp \tau^\perp} u_4 = \mathcal{L}(u_4) - \lambda\partial_{zz} u_4 - \lambda' \partial_z u_4 - \partial_{\tau\tau} u_4. $$
Therefore we can differentiate in the $\tau^\perp$ direction near the lateral boundaries as well.  Thus for any $(x,y,z)$ near $\partial\Omega\times[0,h]$, we have found a basis of directions $(\tau, \tau^\perp, z)$ such that $\partial_{zz}u_3$, $\partial_{\tau\tau}u_3$, and $\partial_{\tau^\perp \tau^\perp}u_3$ all belong to $H^1(\Omega\times(0,h))$, and therefore $\grad^2 u_4 \in H^1(\Omega\times(0,h))$.

We now outline how to obtain higher regularity ($H^s$ for $3<s\leq 5.5$) inductively. The estimate \eqref{subcriticalestimate} yielded $H^{3}$ regularity contingent on the finiteness of the $H^2$ norm of $u_3$. Differentiating this equality again in $z$ and arguing as before gives a finite $H^{4}$ norm of $u_3$. We remark that as in the equality \eqref{subcriticalestimate}, the vanishing of $\lambda'$, $\lambda''$, and $\lambda'''$ eliminates singularities or Dirac deltas at $z=0,h$ which arise when calculating $\mathcal{L}(\frac{\partial^k}{\partial z_k}u_{4,E})$.  Applying the same reasoning another time, we reach $H^{5}$.  For the final half-derivative, $j_E$ runs out of differentiability at order 4, and so we reach $H^{5.5}$.

\end{enumerate}
\end{proof}

\section{Construction of a Smooth Solution}

We begin this section with a technical lemma which will be used to show that under the assumptions on $f$ and $g$ in the statement of \cref{shorttimeexistence}, $\mathcal{L}(\Psi)(t)$ and $\partial_\nu \Psi(t)$ vanish in a neighborhood of $\partial\Omega\times\{0,h\}$ for all $t$.

\begin{lemma}\label{compactsupport}
Let $u:\Omega\times[0,T]\rightarrow \mathbb{R}^2$ be a divergence-free vector field belonging to $L^\infty_t\left(C^1(\bar{\Omega})\right)$ such that $u(x) \cdot \nu_s = 0$ for $x\in\partial\Omega$.  Let $\Gamma(x,t)$ be the solution to the 
\[
\begin{dcases}
  \Dot{\Gamma}(x,t) = u\left( \Gamma(x,t), t \right) \\
  \Gamma(x,t_0) = x
\end{dcases}
\]
for $x\in\bar{\Omega}$ and $t_0\in[0,T]$. Let $\tilde{\Omega}\subset\subset \Omega$.  Then $ \left\{ \Gamma\left(x,t\right) : x\in\tilde{\Omega} \right\} \subset \subset \Omega$ for all $t\in[0,T]$.
\end{lemma}

\begin{proof}

First, by the regularity assumption on $u$ and the vanishing of the normal component of $u$ on $\partial\Omega$, $\Gamma$ is well-defined as the solution to the ODE.  Note that if $x\in\partial\Omega$, $\Gamma(x,t)$ remains in $\partial\Omega$ forwards and backwards in time from $t_0$ since $u|_{\partial\Omega}$ is tangent to $\partial\Omega$. Conversely, it then holds that any point in the interior of $\Omega$ at time $t_0$ remains so under the flow of $\Gamma$.  Consider the function $d(x,t):\tilde{\Omega}\times[0,T] \rightarrow [0,\infty)$ which gives the distance from $\Gamma(x,t)$ to $\partial\Omega$ for $x\in\tilde{\Omega}$.  By the continuity in $x$ of $\Gamma$, for a fixed $t\in[0,T]$, $d(\cdot,t)$ is a continuous function on $\tilde{\Omega}$.  However, we know that $d(\cdot,t)>0$ since $\Gamma$ maps the interior of $\Omega$ to itself.  Since the domain $\tilde{\Omega}$ of $d(\cdot,t)$ is compact, the image of $\tilde{\Omega}$ under $d(\cdot,t)$ is compact in $(0,\infty)$ and therefore has a minimum value which must be strictly larger than $0$.  Therefore, for $x\in\tilde{\Omega}$ the distance from $\Gamma(x,t)$ to $\partial\Omega$ is strictly bounded away from zero, and thus $\tilde{\Omega}$ remains compactly supported in $\Omega$ for all $t\in[0,T]$.

\end{proof}

Throughout the remainder of this section, the notation $C(\Omega)$ indicates a constant which depends on $\Omega$ but may change from line to line (similarly for $h$, $\beta_0$, $\lambda$, etc.).  Constants whose values remain fixed from line to line will be noted.  We aim to build a smooth solution on a short time interval to the system
\[ 
\begin{dcases}
       \left( \partial_t + \grad^\perp\Psi\cdot\grad \right) \left( \mathcal{L}(\Psi) + \beta_0 y \right) =0 &  \Omega\times(0,h)\times[0,T] \\
       \left( \partial_t + \grad^\perp\Psi\cdot\grad \right) \partial_\nu \Psi =0 &  \Omega\times\{0,h\}\times[0,T]\\
       \Psi(x,y,z,t) = \Psi(z,t) & \partial\Omega\times[0,h]\times[0,T] \\
       \dashint_{\partial\Omega\times\{z\}} \grad \Psi(z) \cdot \nu_s = j(z) & [0,h]\times[0,T]
       \end{dcases}
\]
Consider the set of functions 
$$ \mathbb{X}= \left\{ \Psi \in L^\infty\left( [0,T]; \Hilbert \cap H^{5.5} \right) \right\} $$
%\textnormal{ such that } \grad^\perp \partial_\nu \Psi|_{\partial\Omega\times\{0,h\}} = 0  \right\}%
for $T$ to be chosen later.  For $\Psi \in \mathbb{X}$, we define a solution operator $\mathcal{S}$, which maps $\Psi$ to the solution of the linearized version of the system with velocity field $\grad^\perp\Psi$.  Specifically, let $F=F(\Psi)$ and $G=G(\Psi)$ solve
\[ 
\begin{dcases}
       \left( \partial_t + \grad^\perp\Psi\cdot\grad \right) \left( F + \beta_0 y \right) =0 &  \Omega\times(0,h)\times[0,T] \\
       \left( \partial_t + \grad^\perp\Psi\cdot\grad \right) G =0 &  \Omega\times\{0,h\}\times[0,T]\\
       F|_{t=0} = f \\
       G|_{t=0} = g ,
       \end{dcases}
\]
and for $t\in[0,T]$, define $\mathcal{S}(\Psi)(t)$ as the solution to the elliptic problem
\[ 
\begin{dcases}
       \mathcal{L}\left( \mathcal{S}(\Psi)(t) \right) = F(t) &  \Omega\times(0,h) \\
       \partial_{\nu}\left( \mathcal{S}(\Psi)(t) \right) = G(t) & \Omega\times\{0,h\}\\
       \mathcal{S}(\Psi)(x,y,z,t) = \mathcal{S}(\Psi)(z,t) &  \partial\Omega\times[0,h]\\
       \dashint_{\partial\Omega\times\{z\}} \grad \left( \mathcal{S}(\Psi)(t,z) \right) \cdot \nu_s = j(z) & [0,h].
       \end{dcases}
\]

\begin{enumerate}[label=Claim \arabic*:]
\item $\mathcal{S}: \mathbb{X} \rightarrow \mathbb{X} $ is a well-defined mapping.
\begin{proof}
We first note that by the incompressiblity of the flow, the quantities
$$  \int_{\Omega\times(0,h)} F(t), \qquad \int_{\Omega\times\{0,h\}} G(t)  $$
are preserved in time.  Therefore the compatibility condition from \cref{compatibility} is satisfied for all time so that $\mathcal{S}(\Psi)(t)$ is well-defined as the solution to the elliptic problem. We now show that $\mathcal{S}(\Psi)(t)\in H^{5.5}(\Omega\times(0,h))$ for all time $t$. 

Since $F$ solves 
$$ \left(\partial_t+\grad^\perp\Psi\cdot\grad\right)F = -\beta_0\partial_x \Psi, $$
we apply $D^\alpha$ to the equation for $|\alpha| = 4$, multiply by $D^\alpha F$, and integrate by parts to obtain
\begin{align*}
    \frac{1}{2}\frac{\partial}{\partial t} \int_{\Omega\times(0,h)} &\left| D^\alpha F \right|^2 = -\int_{\Omega\times(0,h)} D^\alpha \left( \grad^\perp\Psi \cdot \grad F \right) D^\alpha F - \beta_0 \int_{\Omega\times(0,h)} \partial_x (D^\alpha \Psi) D^\alpha F \\
    &= - \int_{\Omega\times(0,h)} \left[ D^\alpha, \grad^\perp\Psi \cdot \right] (\grad F) D^\alpha F - \beta_0 \int_{\Omega\times(0,h)} \partial_x (D^\alpha \Psi) D^\alpha F\\
    &\leq C\left(\Omega, h, \beta_0\right) \left\| D^\alpha F \right\|_{L^2} \\
    & \qquad \qquad \times \left( \left\| \grad^\perp\Psi\right\|_{C^1} \left\|\nabla^{|\alpha|-1} \grad F \right\|_{L^2} + \| F \|_{L^\infty} \left\| \nabla^{|\alpha|} \grad^\perp \Psi \right\|_{L^2} + \left\| D^\alpha \partial_x \Psi \right\|_{L^2} \right).
\end{align*}
Summing over $\alpha$ and using Sobolev embedding to control $\grad^\perp\Psi$, we obtain that
$$  \frac{\partial}{\partial t} \| F \|_{H^4} \leq C\left( \Omega,h,\beta_0 \right) \left( \left\| \Psi \right\|_{H^5} \| F \|_{H^4} +  \left\| \Psi \right\|_{H^5} \right). $$
Applying Gr\"{o}nwall's inequality gives that for $t\in[0,T]$,
$$ \| F(t) \|_{H^4} \leq C\left( \Omega,h,\beta_0 \right) e^{\int_0^T \left\| \Psi(\tau)\right\|_{H^5}\,d\tau}  \left( \|f\|_{H^4} + \int_0^T \| \Psi(\tau) \|_{H^5}\,d\tau \right). $$

An entirely analogous argument for $G$ yields
\begin{align*}
\| G(t) \|_{{H}^4(\Omega\times\{0,h\})} \leq C\left( \Omega \right) \| g \|_{{H}^4(\Omega\times\{0,h\})} e^{\int_0^T \left\| \grad^\perp\Psi(s) \right\|_{{H}^4(\Omega\times\{0,h\})}\,ds}. 
\end{align*}

Before applying \cref{higherregularity}, we must verify the compatibility conditions from \cref{scompatibility}. Applying \cref{compactsupport}, we deduce that if the support of $f(\cdot,\cdot,z)\subset\subset\Omega$ for fixed $z$, then the support of $F(\cdot,\cdot,z,t)\subset\subset\Omega$.  By the assumption on $f$ in \cref{shorttimeexistence}, for $z$ close enough to $0$ or $h$, the support of $f(\cdot,\cdot,z)\subset\subset \Omega$.  Therefore, $\supp F(\cdot,\cdot,z,t)\subset\subset\Omega$, and thus $\supp F$ remains at positive distance from $\partial\Omega\times\{0,h\}$ for all time. Then $\partial_z F|_{\partial\Omega\times\{0,h\}}=0$, and we have now shown the first condition from \cref{scompatibility}.

To show the second condition, we must show that we can replace the $H^4(\Omega\times\{0,h\})$ norms in the differential equality for $G(t)$ with $\bar{H}^4(\Omega\times\{0,h\})$.  Since $g$ is compactly supported in $\Omega\times\{0,h\}$ by the assumptions of \cref{shorttimeexistence}, applying \cref{compactsupport} shows that $G(t)$ is compactly supported in $\Omega\times\{0,h\}$ for all time.  Therefore,
$$  \left\| G(t) \right\|^2_{\bar{H}^4(\Omega\times\{0,h\})} = \int_{\Omega\times\{0,h\}} \left| \lap^2 G(t) \right|^2 \leq \| G(t) \|^2_{H^4(\Omega\times\{0,h\})}.  $$
Next, we have that due to the continuous inclusion of the domain of $(-\lap)^\alpha$ into the classical Sobolev space $H^\alpha(\Omega)$ for $\alpha\geq0$ (consult \cite{cn2} for example), replacing $\|g\|_{H^4}$ with $\|g\|_{\bar{H}^4}$ on the right hand side can be done immediately without any assumptions on $g$, and we have shown the second condition from \cref{scompatibility}.  

To verify the third compatiblity condition, after appealing to the assumptions on $j$ in \cref{shorttimeexistence}, it suffices to show that 
$$  \frac{\partial^k}{\partial z^k} \int_{\partial\Omega\times\{0,h\}} \grad u_1 \cdot \nu_s = 0  $$
for $k=1,3$ and $u_1$ the solution to
\[
\begin{dcases}
  \mathcal{L}(u_1) = f & \Omega\times(0,h) \\
  \partial_\nu u_1 = g & \Omega\times\{0,h\} \\
  u_1 = 0 & \partial\Omega\times[0,h].
\end{dcases}
\]
For $k=1$, we use the compact support of $G(t)$ to notice that $\partial_z \grad u_1 = 0$ in a neighborhood (in $x$ and $y$) of $\partial\Omega\times\{0,h\}$.  For $k=3$, first note that by the assumption on $\lambda$ in \cref{shorttimeexistence}, 
$$  \frac{\partial}{\partial z}\left( \mathcal{L} -\lap \right)u_1 |_{z=0,h} = \lambda \frac{\partial^3}{\partial z^3} u_1 |_{z=0,h}.  $$
Therefore, we can write that
\begin{align}\label{compatibilityinaction}
\lambda \frac{\partial^3}{\partial z^3} \int_{\partial\Omega\times\{0,h\}} \grad u_1 \cdot \nu_s &= \frac{\partial}{\partial z} \int_{\partial\Omega\times\{0,h\}} \left( \mathcal{L} -\lap \right) \grad u_1 \cdot \nu_s\nonumber \\
& = \int_{\partial\Omega\times\{0,h\}} \left( \frac{\partial}{\partial z} \grad F  - \grad \lap G \right) \cdot \nu_s\nonumber\\
&=0
\end{align}
by the fact that $F$ and $G$ vanish near $\partial\Omega\times\{0,h\}$.  Thus we have verified the third compatibility condition from \cref{scompatibility}.  We remark that this step of the argument is the one of the main reasons that we impose the conditions on $f$, $g$, $j$, and $\lambda$ in the statement of \cref{shorttimeexistence}.

Now we can apply \cref{higherregularity} to give that $\mathcal{S}$ is a self-map of $\mathbb{X}$, and
\begin{align}\label{gronwall}
\left\| \mathcal{S}(\Psi)(t) \right\|_{H^{5.5}} &\leq \tilde{C}(\Omega,h,\lambda,\beta_0) e^{\int_0^T \left\| \Psi(s) \right\|_{H^{5.5}}\,ds} \nonumber\\
&\qquad \left( \|f\|_{H^4} + \|g\|_{\bar{H}^4} + \|j\|_{H^4} + \int_0^T \left\| \Psi(s) \right\|_{H^{5.5}}\,ds \right),
\end{align}
showing that $\mathcal{S}$ maps $\mathbb{X}$ into itself (for any $T$).
\end{proof}

\item There exists a choice of $T_1$ and a set $B \subset \mathbb{X}$ such that $\mathcal{S}:B\rightarrow B$.
\begin{proof}
Define 
$$R=4 \tilde{C}(\Omega,h,\lambda,\beta_0)\left( \| f \|_{H^4} + \| g \|_{\bar{H}^4} + \| j \|_{H^4} \right)$$
where $\tilde{C}(\Omega,h,\lambda,\beta_0)$ is the constant from \eqref{gronwall}, and $B \subset \mathbb{X}$ by
$$ B = \left\{ \Psi \in \mathbb{X} : \| \Psi \|_{\mathbb{X}} \leq R \right\}. $$
We have that $\mathcal{S}(\Psi)|_{t=0}$ is independent of $\Psi$, and 
$$ \| \mathcal{S}(\Psi)(0) \|_{H^{5.5}} \leq \tilde{C}(\Omega,h,\lambda,\beta_0) \left( \| f \|_{H^4} + \| g \|_{\bar{H}^4} + \| j \|_{H^4} \right)= \frac{1}{4}R < R. $$
For $0<T$ and $\Psi \in B$, \eqref{gronwall} shows that
\begin{align*}
   \left\| \mathcal{S}(\Psi)(t) \right\|_{H^{5.5}} &\leq \tilde{C}(\Omega,h,\lambda,\beta_0) e^{\int_0^T \left\| \Psi(s) \right\|_{H^{5.5}}\,ds} \left( \|f\|_{H^4} + \|g\|_{\bar{H}^4} + \|j\|_{H^4} + \int_0^T \left\| \Psi(s) \right\|_{H^{5.5}}\,ds \right)\\
   &\leq \tilde{C}(\Omega,h,\lambda,\beta_0) e^{\int_0^T R\,ds} \left( \|f\|_{H^4} + \|g\|_{\bar{H}^4} + \|j\|_{H^4} + \int_0^T R\,ds \right).
\end{align*}
Since this bound varies continuously in $T$ and is strictly less than $R$ at $T=1$, we can find $T_1>0$ such that for all $t\in[0,T_1]$ and $\Psi \in B$, 
$$ \| \mathcal{S}(\Psi)(t) \|_{H^{5.5}} \leq \frac{1}{3}R < R. $$
To check the size of $T_1$ for large data, we must control the exponential term $e^{\int_0^T R\,ds}$, which given the choice of $R$ remains comparable to 1 for 
$$  T_1 \approx \left( \|f\|_{H^4} + \|g\|_{H^4} + \|j\|_{H^4} \right)^{-1}.  $$
\end{proof}

\item There exists $T_0$ such that $\mathcal{S}$ has a fixed point in $B$.
\begin{proof}
Define $\Psi^{(0)}(t,x,y,z) = \Psi_0(x,y,z)$, and define the sequence of functions 
$$ \Psi^{(n)} = \mathcal{S}(\Psi^{(n-1)}) $$
inductively. Since $\mathcal{S}$ is a self-map of $B$, $\Psi^{(n)}$ is well-defined for all $n\in\mathbb{N}$. We claim that for a suitable choice of $T_0$, $\nabla \Psi^{(n)}$ is a Cauchy sequence in the space 
$$L^\infty\left([0,T_0];H^\frac{1}{2}(\Omega\times(0,h))\right).$$
Let integers $n$ and $k$ be fixed.  Then $\mathcal{L}(\Psi^{(n+k)}) - \mathcal{L}(\Psi^{(n)})$ satisfies the equation
\begin{align*}
\partial_t \left( \mathcal{L}(\Psi^{(n+k)}) - \mathcal{L}(\Psi^{(n)}) \right) &+ \grad^\perp\Psi^{(n+k-1)} \cdot \grad \left( \mathcal{L}(\Psi^{(n+k)}) - \mathcal{L}(\Psi^{(n)}) \right)\\
&\qquad = \left( \grad^\perp\Psi^{(n-1)} - \grad^\perp\Psi^{(n+k-1)} \right) \cdot \grad \left(\mathcal{L}(\Psi^{(n)}) + \beta_0 y \right).
\end{align*}
Multiplying by $\mathcal{L}(\Psi^{(n+k)}) - \mathcal{L}(\Psi^{(n)})$, integrating by parts, and using Gronw\"{a}ll's inequality again shows that for $t\in[0,T_0]$,
\begin{align*}
    &\left\| \mathcal{L}(\Psi^{(n+k)})(t) - \mathcal{L}(\Psi^{(n)})(t) \right\|_{L^2(\Omega\times(0,h))}\\
    &\qquad\leq \int_0^{T_0} \left\| \left( \grad^\perp\Psi^{(n-1)}(\tau) - \grad^\perp\Psi^{(n+k-1)}(\tau) \right) \cdot \grad (\mathcal{L}\Psi^{(n)}(\tau)+\beta_0 y) \right\|_{L^2(\Omega\times(0,h))} \,d\tau\\
    &\qquad < T_0 \bigg{(} \left\| \left(\nabla\Psi^{(n-1)}-\nabla\Psi^{(n+k-1)}\right)\right\|_{L^\infty([0,T_0];L^2(\Omega\times(0,h)))} \\
    &\qquad\qquad\qquad\qquad \times \left\|\grad\left(\mathcal{L}\Psi^{(n)}+\beta_0 y\right) \right\|_{L^\infty([0,T_0];L^\infty(\Omega\times(0,h)))} \bigg{)}\\
    &\qquad \leq T_0 (R+\beta_0) \left\| \left(\nabla\Psi^{(n-1)}-\nabla\Psi^{(n+k-1)}\right)\right\|_{L^\infty([0,T_0];L^2(\Omega\times(0,h)))}.
\end{align*}
A completely analogous argument holds for $\partial_\nu \Psi^{(n)}$. Solving the elliptic problem and summing then shows that 
\begin{align*}
    &\left\| \left(\nabla\Psi^{(n)}-\nabla\Psi^{(n+k)}\right)\right\|_{L^\infty([0,T_0];H^\frac{1}{2}(\Omega\times(0,h)))} \\
    &\qquad \leq 2 T_0 (R+\beta_0) C(\Omega,h,\lambda) \left\| \left(\nabla\Psi^{(n-1)}-\nabla\Psi^{(n+k-1)}\right)\right\|_{L^\infty([0,T_0];H^\frac{1}{2}(\Omega\times(0,h)))}\\
    &\leq \frac{1}{2} \left\| \left(\nabla\Psi^{(n-1)}-\nabla\Psi^{(n+k-1)}\right)\right\|_{L^\infty([0,T_0];H^\frac{1}{2}(\Omega\times(0,h)))}
\end{align*}
if $T_0$ is chosen to absorb the constant $2(R+\beta_0)C(\Omega,h,\lambda)$. Iteratively applying this bound then shows that $\nabla\Psi^{(n)}$ is a Cauchy sequence as desired.  In addition, the choice of $R$ gives the desired lower bound on $T_0$ for large data.  

Since $\Psi^{(n)}$ converges strongly to $\Psi$ in $L^\infty\left([0,T_0];H^{\frac{1}{2}}(\Omega\times(0,h))\right)$ and is uniformly bounded in $L^\infty\left([0,T_0];H^{5.5}(\Omega\times(0,h))\right)$, interpolation gives that $\Psi^{(n)}$ converges strongly in $L^\infty\left([0,T_0];H^{s}(\Omega\times(0,h))\right)$ for $s<5.5$. Then define $\Psi(t)$ to be the solution to the elliptic problem
\[ 
\begin{dcases}
       \mathcal{L}\left( \Psi(t) \right) = \lim_{n\rightarrow\infty} \mathcal{L}\left( \Psi^{(n)}(t) \right) &  \Omega\times(0,h) \\
       \partial_{\nu}\left( \Psi(t) \right) = \lim_{n\rightarrow\infty} \partial_{\nu}\left( \Psi^{(n)}(t) \right) & \Omega\times\{0,h\}\\
       \Psi(x,y,z,t) = \Psi(z,t) &  \partial\Omega\times[0,h]\\
       \dashint_{\partial\Omega\times\{z\}} \grad \left( \Psi(t,z) \right) \cdot \nu_s = j(z) & [0,h].
       \end{dcases}
\]
Passing to the limit in the QG equations, we have therefore shown that $\Psi$ is a fixed point of $\mathcal{S}$.

%We must still remove the assumption that the support of $g$ is compactly contained in $\Omega\times\{0,h\}$.  By the assumption of \cref{scompatibility}, we can take a sequence of functions $g_n \in \mathcal{D}(\Omega)$ converging to $g$ in $\bar{H}^4(\Omega\times\{0,h\})$ and apply the existence result for $g_n$ to produce solutions $\Psi^{(n)}$.  Standard energy estimates as before then give stability of the sequence $\Psi^{(n)}$ in Sobolev spaces which are strong enough to pass to the limit $\Psi^{(n)}\rightarrow \Psi$.
\end{proof}

\end{enumerate}

\section{Appendix}

\begin{prop}[\textbf{Commutator Estimate}]\label{commutatorestimate}
For $f,g: \Omega\times(0,h) \rightarrow \mathbb{R}$, there exist constants $C(\Omega,h,s)$ such that for $\alpha$ a multi-index with $|\alpha|=s$, 
$$  \| D^\alpha(fg) - f D^\alpha(g) \|_{L^2\left(\Omega\times(0,h)\right)} \leq C(\Omega,h,s) \left( \| \nabla f \|_{L^\infty} \|\nabla^{(s-1)}g\|_{L^2} + \|g\|_{L^\infty}\|\nabla^s f\|_{L^2} \right).  $$
\end{prop}
\begin{proof}
Substituting $\Omega\times(0,h)$ for $\mathbb{T}^n$, the statement is precisely the Klainerman-Majda commutator estimate from \cite{km81}.  The ingredients of the proof in that case are H\"{o}lder's inequality and the Gagliardo-Nirenberg interpolation inequality.  As H\"{o}lder's inequality is valid for $\Omega\times(0,h)$, we can follow the classical proof provided that the Gagliardo-Nirenberg inequality holds for $\Omega\times(0,h)$.  Since $\Omega\times(0,h)$ is a bounded domain Lipschitz domain, Stein's linear Sobolev extension operator $\mathcal{E}$ \cite{ste70} gives that for $k\in\mathbb{N}$ and $1\leq p \leq \infty$,
$$ \mathcal{E}: W^{k,p}\left( \Omega\times(0,h) \right) \rightarrow W^{k,p} \left(  \mathbb{R}^3 \right) $$
is bounded with constants depending only on $k$, $p$, $\Omega$, and $h$. Utilizing the extension, it is simple to show that Gagliardo-Nirenberg holds for $\Omega\times(0,h)$, completing the proof.
\end{proof}

\bibliography{references.bib}
\bibliographystyle{plain}
\nocite{cv}
\nocite{bsv16}
\nocite{cvicol}
\nocite{knv}
\nocite{pv}
\nocite{cmt}
\nocite{Marchand}
\nocite{Resnick}
\nocite{dg}
\nocite{bb}
\nocite{novackvasseur}
\nocite{a}
\nocite{novackweak}
\nocite{ci}
\nocite{ci2}
\nocite{cn}
\nocite{cn2}
\end{document}